\documentclass{article}
\usepackage{float}
\usepackage{graphicx,epstopdf}
\usepackage{bbm}
\usepackage{color}
\usepackage{amsfonts}
\usepackage{mathrsfs}
\usepackage{amsmath}
\usepackage{dsfont}
\usepackage{amssymb}
\usepackage{csquotes}
\usepackage{theorem}
\usepackage{graphicx}
\usepackage[dvips]{epsfig}
\usepackage{latexsym}
\usepackage{exscale}
\usepackage[latin1]{inputenc}
\newtheorem{theorem}{Theorem}

\newtheorem{definition}[theorem]{Definition}
\newtheorem{example}[theorem]{Example}

\newtheorem{lemma}[theorem]{Lemma}
\newtheorem{notation}[theorem]{Notation}

\newtheorem{remark}[theorem]{Remark}

\newenvironment{proof}[1][Proof]{\noindent\textbf{#1.} }{\ \rule{0.5em}{0.5em}}
\begin{document}

\title{ Extensions of Bernstein's Lethargy Theorem}
\author{Asuman G\"{u}ven Aksoy}

\date{\vspace{-5ex}}
\maketitle

\maketitle
\begin{abstract}
In this   paper,  we examine the  aptly-named ``Lethargy Theorem" of Bernstein and survey its recent extensions. We  show that  one of these extensions shrinks the interval for best approximation by half while the other gives a surprising connection to the  space of bounded linear operators between two Banach spaces.

\end{abstract}

\section{Introduction}
The formal beginnings of approximation theory date back to 1885, with Weierstrass' celebrated approximation theorem. The discovery that every continuous function defined on a closed interval $[a,b]$ can be uniformly approximated as closely as desired by a polynomial function immediately prompted many new questions. One such question concerned approximating functions with polynomials of limited degree. That is, if we limit ourselves to polynomials of degree at most $n$, what can be said of the best approximation? As it turns out, there is no unified answer to this question. In fact, S. N. Bernstein \cite{Bernstein} in 1938  showed that there exist functions whose best approximation converges arbitrarily slowly as the degree of the polynomial rises. In this paper, we take up this aptly-named ``Lethargy Theorem" of Bernstein and present two extensions.
For $f \in C([0,1])$,  the sequence of the best approximations is defined as:
\begin{equation}
\rho(f, P_n)= \inf \{ || f-p||:\,\,\, p\in P_n\}
\end{equation}
where $P_n$ denotes the space of all polynomials of degree $\leq n$.
Clearly,  $$ \rho(f, P_1) \geq\rho(f,P_2)\geq \cdots $$ and  $ \{\rho(f,P_n)\}$ form a non-increasing sequence of numbers.  Bernstein \cite{Bernstein} proved that if $\{d_n \}_{n\ge1}$ is a non-increasing null sequence (i.e. $\lim\limits_{n\to\infty}d_n=0$) of positive numbers, then there exists a function $f\in C[0,1]$ such that
$$\rho(f, P_n)= d_n,~\mbox{for all $n \ge1$}.$$
This remarkable result is called Bernstein's Lethargy Theorem (BLT) and is used in the constructive theory of functions \cite{Sin}, and it has been applied to the theory of quasi analytic functions in several complex variables \cite{Ple2, Ple3}.  Also see \cite{II} and references therein for an application of BLT to the study  Gonchar quasianalytic functions of several variables. 

 Bernstein's proof is based on compactness argument and only works when the subspaces  are finite dimensional. 
Note that the density of polynomials in $C[0,1]$ (the Weierstrass Approximation Theorem) implies that $$\displaystyle \lim_{n \to \infty} \rho(f, P_n) = 0.$$ However, the Weierstrass Approximation Theorem gives no information about the speed of convergence for $\rho(f, P_n)$, but Bernstein's Lethargy Theorem does.   Bernstein Lethargy Theorem has been extended, replacing $C[0,1]$ by an arbitrary Banach space $X$ and replacing  $P_n$ by arbitrary closed subspaces $\{Y_n\}$  of $X$  where $Y_1 \subset Y_2 \subset \cdots$  by Shapiro \cite{Sha}. Using Baire's category theory, he proved that  for each null sequence $\{d_n\}$ of non-negative numbers, there exists a vector $x\in X$ such that
\begin{equation*}
\label{Shap}
 \rho(x, Y_n) \neq O (d_n),~\mbox{as $n\to\infty$}.
 \end{equation*}
That is, there is no $M>0$ such that
$$\rho(x, Y_n) \leq M d_n,~\mbox{ for all $n\ge1$}.
 $$

Note that Shapiro's result is not restricted to  finite dimensional subspaces. 
 This result  was later strengthened by Tyuriemskih \cite{Tyu} . He showed that the sequence of errors of the best approximation from $x$ to $Y_n$, $\{\rho(x,Y_n)\}$ may converge to zero at an arbitrary slow rate. More precisely, for any expanding sequence $\{Y_n\}$ of subspaces of $X$ and for any null sequence $\{d_n\}$ of positive numbers, he constructed an element $x\in X$  such that
 \begin{equation*}
 \label{tyu}
 \lim_{n \rightarrow \infty} \rho(x, Y_n) =0,~\mbox{and}~\rho(x, Y_n) \geq d_n~\mbox{for all $n\ge1$}.
   \end{equation*}
   However, it is also possible that the errors of the best approximation $\{\rho(x,Y_n)\}$ may converge to zero arbitrarily fast; for results of this type see   \cite{Al-To}.
   
    We refer the reader to \cite{De-Hu} for an application of Tyuriemskih's Theorem to convergence of sequence of bounded linear operators. 
    
     We also refer to \cite{Ak-Al,Ak-Le,Al-To,Lew1} for other versions of Bernstein's Lethargy Theorem  and  to \cite{Ak-Le2,Alb,Mich,Ple1,Vas} for Bernstein's Lethargy Theorem for Fr\'{e}chet  spaces. 
    
    Given an arbitrary Banach space $X$, a strictly increasing sequence $\{Y_n \}$ of subspaces of $X$  and a non-increasing null sequence $\{d_n\}$ of non-negative numbers, one can ask the question whether  there exists $x \in X$ such that $\rho(x, Y_n) = d_n $ for each $n$? For a long time, no sequence $\{d_n\}$ of this type was known for which such an element $x$ exists for \textit{all} possible Banach spaces $X$. The only known spaces $X$ in which the answer is always ``yes" are the Hilbert spaces (see \cite{Tyu2}). For a general (separable) Banach space $X$, a solution $x$ is known to exist whenever all $Y_n$  are finite-dimensional (see \cite{Tim}).
Moreover, it is known that  if $X$ has  the above property, then it is reflexive (see \cite{Tyu2}).
\section {Extensions of BLT to Banach spaces}
The following lemma is a   Bernstein's Lethargy type  of a result restricted to  \textit{finite number} of subspaces,   and for the proof of this lemma  we refer the reader to Timan's book  \cite{Tim}.
\begin{lemma}
\label{lemma1}
Let $(X,\|\cdot\|)$ be a normed linear space, $Y_1\subset Y_2\subset\ldots\subset Y_n\subset X$ be a finite system of strictly nested subspaces, $d_1>d_2>\ldots>d_n\ge0$ and $z \in X\backslash Y_n$. Then, there is an element $x\in X$ for which $\rho(x,Y_k)=d_k$ $(k=1,\ldots,n)$,  $\|x\|\le d_1+1$, and $x-\lambda z\in Y_n$ for some $\lambda>0$.
\end{lemma}
 Borodin \cite{Borodin1} proved  the above Lemma \ref{lemma1} by taking $(X,\|\cdot\|)$  to be a Banach space.  Returning to the question posed before, namely given an arbitrary Banach space $X$, a strictly increasing sequence $\{Y_n \}$ of subspaces of $X$  and a non-increasing null sequence $\{d_n\}$ of non-negative numbers, one can ask the question whether  there exists $x \in X$ such that $\rho(x, Y_n) = d_n $ for each $n$? For a long time no sequence $\{d_n\}$ of this type was known for which such an element $x$ exists for \textit{all} possible Banach spaces $X$.
 Borodin  in \cite{Borodin1} uses the above lemma for Banach space to establish the existence of such an element in case of rapidly decreasing sequences; more precisely, in 2006 he proves the following theorem:
\begin{theorem}[Borodin \cite{Borodin1}]
  \label{thm:Borodin}
  Let $X$ be an arbitrary Banach space (with finite or infinite dimension),  $Y_1 \subset Y_2
  \subset \cdots$ be an arbitrary countable system of strictly nested subspaces
  in $X$, and fix a numerical sequence $\{d_n\}_{n\ge1}$ satisfying: there exists a natural number $n_0\ge1$ such that
  \begin{equation}
  \label{condition_Borodin}
  d_n >
  \sum_{k = n+1}^\infty d_k~\mbox{for all $n \ge n_0$ at which $d_n > 0$.}
  \end{equation}
  Then there is an element $x \in X$ such that
  \begin{equation}
  \label{rhod}
  \rho(x,Y_n) = d_n, ~\mbox{for all $n\ge1$}.
  \end{equation}
\end{theorem}
The condition (\ref{condition_Borodin}) on the sequence $\{d_n\}$ is the key to the derivation of (\ref{rhod}) in Theorem \ref{thm:Borodin}. Based on this result, Konyagin \cite{Konyagin} in $2013$ takes a further step to show that, for a general non-increasing null sequence $\{d_n\}$, the deviation of $x\in X$ from each subspace $Y_n$ can range in some interval depending on $d_n$.
\begin{theorem}[Konyagin \cite{Konyagin}]
\label{thm:Konyagin}
Let $X$ be a real Banach space, $Y_1 \subset Y_2
  \subset \cdots$  be a sequence of strictly nested closed linear  subspaces of $X$, and $d_1 \geq d_2 \geq \cdots$  be a non-increasing sequence converging to zero, then there exists an element $x\in X$ such that
the distance $\rho(x, Y_n)$ satisfies the inequalities
  \begin{equation}
  \label{bound_Kongyagin}
   d_n \leq \rho(x, Y_n) \leq 8d_n, \,\,\,\mbox{for $n\ge1$}.
   \end{equation}
\end{theorem}
Note that the condition (\ref{condition_Borodin}) is satisfied when $d_n =(2+\epsilon)^{-n} $ for $\epsilon > 0$ arbitrarily small, however it is not satisfied when $d_n= 2^{-n}$. Of course there are two natural questions to ask:
 \begin{description}
   \item[\textbf{Question $1$}] Is the condition (\ref{condition_Borodin}) necessary for the results in Theorem \ref{thm:Borodin} to hold, or does Theorem \ref{thm:Borodin} still hold for the sequence $d_n=2^{-n}$, $n\ge1$?
   \item[\textbf{Question $2$}] Under the same conditions given in Theorem  \ref{thm:Konyagin}, can the lower and upper bounds of $\rho(x, Y_n)$ in (\ref{bound_Kongyagin}) be improved?
 \end{description}
 Both of these questions are answered  and  two improvements on a theorem of S. N. Bernstein for Banach spaces are presented in \cite{Ak-Pe2}.   A positive answer to Question $1$ is obtained in \cite{Ak-Pe2} by showing that Theorem \ref{thm:Borodin} can be extended by weakening the strict inequality in (\ref{condition_Borodin}) to a non-strict one:
\begin{equation}
 \label{condition}
 d_n\ge\sum_{k=n+1}^\infty d_k,~\mbox{for every} \,\, n \ge n_0.
 \end{equation}
 Clearly, the condition  (\ref{condition}) is weaker than (\ref{condition_Borodin}),  but unlike the condition (\ref{condition_Borodin}), (\ref{condition}) is satisfied by the sequences $\{d_n\}_{n\ge1}$ verifying $d_n=\sum\limits_{k=n+1}^\infty d_k$ for all $n\ge n_0$. For a typical example of such a sequence one can take $\{d_n\} =\{2^{-n}\}$.  
 We have also  shown that if $X$ is an arbitrary infinite-dimensional Banach space,  $\{Y_n\}$ is a sequence of strictly nested subspaces of $ X$,  and if   $\{d_n\}$ is a non-increasing sequence of non-negative numbers tending to 0, then for any $c\in(0,1]$
  we can find  $x_{c} \in X$,    such that the distance $\rho(x_{c}, Y_n)$ from $x_{c}$ to $Y_n$ satisfies
$$
c d_n \leq \rho(x_{c},Y_n) \leq 4c d_n,~\mbox{for all $n\in\mathbb N$}.
$$
We prove the above inequality by first improving Borodin's (2006) result for Banach spaces  by weakening his condition on the sequence $\{d_n\}$.  The weakened condition on $d_n$ requires refinement of Borodin's construction to extract  an element in $X$, whose distances from the nested subspaces are  precisely the given values $d_n$.

Now, we are ready to state the following  theorem \cite{Ak-Pe2} which  improves the theorem  of Borodin \cite{Borodin1}.
\begin{theorem}[Aksoy, Peng \cite{Ak-Pe2} ]
\label{Borodin}
Let $X$ be an arbitrary infinite-dimensional Banach space, $\{Y_n\}_{n\ge1}$ be an arbitrary system of strictly nested subspaces with the property $\overline Y_n \subset Y_{n+1}$ for all $n\ge1$, and let the non-negative numbers $\{d_n\}_{n\ge1}$ satisfy the following property: there is an integer $ n_0 \ge 1 $ such that
$$
 d_n\ge\sum_{k=n+1}^\infty d_k,~\mbox{for every} \,\, n \ge n_0.
$$
 Then, there exists an element $x\in X$ such that $\rho(x,Y_n)=d_n$ for all $n\ge1$.
\end{theorem}

Proof of this theorem depends on some technical lemmas, details can be found in \cite{Ak-Pe2}. Next we state an improvement of the Konyagin's result.
\begin{theorem} [Aksoy, Peng \cite{Ak-Pe2}]
\label{Konyagin}
Let $X$ be an infinite-dimensional Banach space, $\{Y_n\}$ be a system of strictly nested subspaces of $X$ satisfying the condition $\overline{Y}_n \subset Y_{n+1}$ for all $n\ge1$. Let $\{d_n\}_{n\ge1}$ be a non-increasing null sequence of strictly positive numbers. Assume that there exists an extension $\{(\tilde d_n,\tilde Y_n)\}_{n\ge1}\supseteq \{(d_n,Y_n)\}_{n\ge 1}$ satisfying: $\{\tilde d_n\}_{n\ge1}$ is a non-increasing null sequence of strictly positive values, $\overline{\tilde Y}_n\subset \tilde Y_{n+1}$ for $n\ge1$; and there is an integer $i_0\ge1$ and a constant $K>0$ such that
$$
\{K2^{-n}\}_{n\ge i_0}\subseteq \{\tilde d_n\}_{n\ge 1}.
$$ Then, for any $c\in(0,1]$, there exists an element $x_c\in X$ (depending on $c$) such that
\begin{equation}
\label{xc}
cd_n\le \rho(x_c,Y_n)\le 4c d_n,~\mbox{for $n\ge1$}.
\end{equation}
\end{theorem}
\textbf{Proof.} We first show (\ref{xc}) holds for $c=1$.

By assumption, there is a subsequence $\{n_i\}_{i\ge i_0}$ of $\mathbb N$ such that
$$
\tilde d_{n_i}=K2^{-i},~\mbox{for $i\ge i_0$}.
$$
Since the sequence $\{\tilde d_n\}_{n=1,2,\ldots,n_{i_0}-1}\cup\{\tilde d_{n_i}\}_{i\ge i_0}$ satisfies the condition (\ref{condition}) and $\overline{\tilde Y}_{n}\subset \tilde Y_{n+1}$ for all $n\ge1$, then we can apply Theorem \ref{Borodin} to get $x\in X$ so that
\begin{equation}
\label{distYd}
\rho(x,\tilde Y_n)=\tilde d_n,~\mbox{for $n=1,\ldots,n_{i_0}-1$},~\mbox{and}~\rho(x,\tilde Y_{n_i})=\tilde d_{n_i},~\mbox{for all $i\ge i_0$}.
\end{equation}
Therefore for any integer $n\ge1$,
\begin{description}
\item[Case 1] if $n\le n_{i_0}-1$ or $n=n_i$ for some $i\ge i_0$,  then  it follows from (\ref{distYd}) that
$$
\rho(x,\tilde Y_n)= \tilde d_n;
$$
\item[Case 2] if $n_i<n< n_{i+1}$ for some $i\ge i_0$, then the facts that $\{\tilde d_n\}$ is non-increasing and $\tilde Y_{n_i}\subset \tilde Y_{n}\subset \tilde Y_{n_{i+1}}$ lead to
$$
\rho(x,\tilde Y_n)\in\left(\rho(x,\tilde Y_{n_{i+1}}),\rho(x,\tilde Y_{n_i})\right)=\left(K2^{-(i+1)},K2^{-i}\right)
$$
and
$$
\tilde d_n\in\left[K2^{-i},K2^{-i+1}\right].
$$
It follows that
$$
\frac{\rho(x,\tilde Y_n)}{\tilde d_n}\in\left(\frac{K2^{-i-1}}{K2^{-i+1}},\frac{K2^{-i}}{K2^{-i}}\right)=\left(\frac{1}{4},1\right).
$$
\end{description}
Putting the above cases  together  yields
$$
\frac{1}{4}\tilde d_n\le \rho(x,\tilde Y_n)\le \tilde d_n~\mbox{for all $n\ge1$}.
$$
For $c\in(0,1]$, taking $x_c=4cx$ in the above inequalities, we obtain
$$
c\tilde d_n\le \rho(x_c,\tilde Y_n)\le 4c\tilde d_n,~\mbox{for all $n\ge1$}.
$$
Remembering that $\{(d_n,Y_n)\}_{n\ge1}\subseteq\{(\tilde d_n,\tilde Y_n)\}_{n\ge1}$, we then necessarily have
\begin{equation*}
\label{notequal}
c d_n\le \rho(x_c, Y_n)\le 4c d_n,~\mbox{for all $n\ge1$}.
\end{equation*}
Hence Theorem \ref{Konyagin} is proved. 

\begin{remark}
 Taking $c=\displaystyle \frac{1}{4}$ in Theorem \ref{Konyagin}, we obtain existence of $x\in X$ for which
  $$
\frac{1}{4}\le \frac{\rho(x,Y_n)}{d_n}\le 1,~\mbox{for all}~n\ge1.
$$
The interval length $\displaystyle \frac{3}{4}$ makes $\big[\displaystyle \frac{1}{4},1\big]$ the \enquote{narrowest} estimating interval of $\displaystyle \frac{\rho(x,Y_n)}{d_n}$ that Theorem \ref{Konyagin} could provide.
\end{remark}
 
 The subspace condition given in Theorem \ref{Konyagin} states that the nested sequence $\{Y_n\}$ has \enquote{enough gaps} so that the sequence $$\{(d_n',Y_n')\}_{n\ge1}=\{(d_n,Y_n)\}_{n\ge1}\cup\{(K2^{-i},\tilde Y_{n_i})\}_{i\ge i_0}$$ satisfies $d_n'\ge d_{n+1}'\to 0$ and $\overline {Y_n'}\subset Y_{n+1}'$ for all $n\ge1$.  For another subspace condition see the results in \cite{Ak-Pe}.
 
 Observe that  in Konyagin's paper \cite{Konyagin} it is assumed that $\{Y_n\}$ are closed
and strictly increasing. In Borodin's paper \cite{Borodin1}, this is not specified, but from the proof of Theorem \ref{thm:Borodin} it is clear that his proof  works only under the assumption that $\overline{Y}_n $ is strictly included
in $Y_{n+1}$. 
The subspace condition $\overline{ Y_n} \subset Y_{n+1}$ does not come at the expense of our assumption to weaken the condition  on  the sequence $d_n$. This  is a natural condition. To clarify the reason why almost all Lethargy theorems have this condition on the subspaces,  we give the following example.
\begin{example}\label{EXAM}
Let $X = L^{\infty}[0,1]$ and consider $C[0,1] \subset L^{\infty} [0,1]$ and define the subspaces  of $X$ as follows:

\begin{enumerate}
\item $Y_1 =  \mbox{space of all polynomials}$;

\item $Y_{n+1} = $span$[Y_n \cup \{f_{n}\}]$
where $f_{n} \in C[0,1] \setminus Y_n$, for $n\ge1$.
\end{enumerate}
Observe that by the Weierstrass Theorem we have $\overline {Y}_n = C[0,1] $ for all $n \ge1$.
Take any  $ f \in L^{\infty}[0,1]$ and consider the following cases:
\begin{enumerate}
\renewcommand{\labelenumi}{\alph{enumi})}
\item If  $ f \in C[0,1]$ , then
$$
\rho (f,Y_n) = \rho(f,C[0,1]) = 0~\mbox{for all $n\ge1$}.
$$
\item If $ f \in L^{\infty}[0,1] \setminus C[0,1] $, then
$$
\rho(f, Y_n) = \rho(f, C[0,1]) = d > 0~ \mbox{ (independent of $n$)}.
$$
Note that in above, we have used the fact that $\rho(f,Y_n) = \rho(f,\overline{Y}_n)$.
Hence in this case BLT does not hold.
\end{enumerate}
\end{example}

\section{Extension of BLT to Fr\'{e}chet  spaces}


Fr\'{e}chet  spaces are locally convex spaces that are complete with respect to a translation invariant metric, and  they are generalization of Banach spaces which are normed linear spaces,  complete with respect to the metric induced by the norm.  However, there are metric spaces which are not normed spaces. This can be easily seen by considering the space  $s$  the set of all sequence $x=(x_n)$ and defining $ d(x,y) =\displaystyle \sum_{i=1}^{\infty}  2^{-i} \displaystyle \frac{|x_i-y_i|}{(1+ |x_i-y_i|)}$  as a metric on $s$.
If we let  $\lambda \in \mathbb{R}$ then
$$d( \lambda x, \lambda y) = \sum_{i=1}^{\infty} \frac{1}{2^i} \frac{| \lambda |  |x_i-y_i|}{(1+| \lambda | |x_1-y_i|)} \neq|\lambda| d(x,y)\quad (\mbox{no homogeneity)}.$$

\begin{definition}

 $(X, \| \cdot \|_F)$ is called a \textit{Fr\'{e}chet  space}, if it is a metric linear space which is complete with respect  to 
 its F-norm $||.||_F$ giving the topology. An  F-norm  $||.||_F$   satisfies the following conditions  \cite{Rol}:
 \begin{enumerate}
 \item $||x||_F =0$ if and only if $x=0$,
 \item $||\alpha x||_F = ||x||_F$ for all real or complex  $\alpha$  with $|\alpha|=1$ (no homogeneity),
 \item $||x+y||_F \leq ||x||_F+||y||_F$,
 \end{enumerate}
 
 \end{definition}
Many Fr\'{e}chet spaces  $X$ can also be constructed using  a countable family of semi-norms $||x||_k$ where $X$ is a complete space with respect to this family of semi-norms. For example a translation invariant complete metric inducing the topology on $X$ can be defined as $$d(x,y) = \displaystyle \sum_{k=0}^{\infty} 2^{-k} \displaystyle \frac{||x-y||_k}{1+||x-y||_k}\,\,\mbox{ for}\,\,\, x,y \in X.$$
Clearly, every Banach space is a Fr\'{e}chet space, and the other well known example of a Fr\'{e}chet space is the vector space $C^{\infty}[ 0,1]$ of all infinitely differentiable functions $f:[0,1] \to \mathbb{R}$ where the semi-norms are $$||f||_k = \sup \{|f^{(k)} (x) |: \,\, x\in [0,1]\}.$$ For more information about Fr\'echet spaces the reader is referred to \cite{Rol} and \cite{Kal}.

 Recently in \cite{Ak-Le2}  a version of  Bernstein Lethargy Theorem (BLT) was given for Fr\'{e}chet spaces.  More precisely, let $X$ be an infinite-dimensional Fr\'echet space and let $\mathcal{V}=\{V_n\}$ be a nested sequence of subspaces of $ X$ such that $ \overline{V_n} \subseteq V_{n+1}$ for any $ n \in \mathbb{N}$. 
Let $ d_n$ be a decreasing sequence of positive numbers tending to 0. Under an additional natural condition on $\sup\{\mbox{dist}(x, V_n)\}$, we proved that there exists $ x \in X$ and 
$ n_o \in \mathbb{N}$   such that
\[
\frac{d_n}{3} \leq \rho (x,V_n) \leq 3 d_n,
\]
for any $ n \geq n_o$. By using the above theorem,  it is also possible to obtain an extension of  both Shapiro's \cite{Sha} and  Tyuremskikh's \cite{Tyu}  theorems for  Fr\'{e}chet  spaces as well.   

\begin{notation}
Let $(X, \|.\|_F)$ be a Fr$\acute{e}$chet  space and assume that $\mathcal{V} = \{ V_n\}$ is a nested sequence of linear subspaces of $X$ satisfying
$\overline{V_n} \subset V_{n+1}$.  Let  $d_{n, \mathcal{V}} $  denote the deviation of $V_n$ from $V_{n+1}$ defined as:

 \begin{equation}
\label{dn}
d_{n, \mathcal{V}} = \sup \{ \rho(v, V_n) : v \in V_{n+1}\}
\end{equation}
\end{notation}
and 
throughout this paper we assume:
 \begin{equation}
 \label{nondegenerate}
d_{\mathcal{V}} = \inf\{ d_{n,\mathcal{V}}: n \in \mathbb{N} \} > 0.
\end{equation}
The necessity of this assumption is illustrated in the following example.
\begin{example}
 Let 
$ 
X  = \{ (x_n): x_n \in \mathbb{R} \hbox{ for any }  n \in \mathbb{N}\} 
$
equipped with the $F$-norm $ \| x\|_F = \displaystyle \sum_{j=1}^{\infty} \frac{|x_j|}{2^j(1+ |x_j|)}, $ where $ x=(x_1,...,x_j,..).$ Let 
$$ 
V_n = \{ x \in X: x_k=0 \hbox{ for } k \geq n+1 \}. 
$$
It is easy to see that for any $ x \in X$ 
$$ 
\rho(x,V_n) = \sum_{j=n+1}^{\infty} \frac{|x_j|}{2^j(1+ |x_j|)}  \leq \frac{1}{2^n}.
$$
Let $ d_n = \displaystyle \frac{2}{n}$ and observe that for any $ x \in X,$ 
$
\rho(x,V_n) \leq  \displaystyle \frac{1}{2^n}< \displaystyle \frac{2}{n}.
$ 
Also observe that $d_{n, {\mathcal V}} =\displaystyle  \frac{1}{2^{n+1}} $ which implies that $ d_{{\mathcal V}}=0.$ 
\end{example}
 Thus in  the case when  $ d_{{\mathcal V}}=0$,  we cannot even hope to prove Shapiro's theorem. Above example  also shows that it is impossible to prove the Tyuriemskih Theorem or Konyagin's  type result  in Fr\'echet spaces without additional assumptions, because they are stronger statements than Shapiro's theorem. 
Note that if $X$ is a Banach space,  then the  condition $d_{\mathcal{V}} = \inf\{ d_{n,\mathcal{V}}: n \in \mathbb{N} \} > 0
$ is satisfied automatically. It can be seen easily that $d_{n, \mathcal{V}} = + \infty$ for Banach spaces. because
$$\rho(tx, V_n) = t\rho(x, V_n)$$ and  the supremum taking over all $v\in V_{n+1}$ and $V_n$ is strictly included in $V_{n+1}$.
The next,  example illustrates that  there is a natural way to build Fr\'{e}chet  spaces where  $d_{n, \mathcal{V}} =1$ .
\begin{example}
Let $(X, \|.\|)$ be a Banach space. Define in X an F-norm $ \| \| _F$ by:
$\|x\|_F = \|x\|/(1+\|x\|)$. Then $d_{n, \mathcal{V}} =1$ for any $n \in \mathbb{N}$ independently of
$\mathcal{V}$. Because the mapping $$t \mapsto \displaystyle \frac{t}{1+t}$$ is increasing for $t> -1$ and 
$$\rho_{F} (tx, V_n) =\displaystyle \frac{\rho(tx, V_n)}{1+\rho(tx, V_n)} \rightarrow 1$$ as $t\rightarrow \infty$.

\end{example}
\begin{theorem}[Aksoy, Lewicki, \cite{Ak-Le2}]\label{Equ}
Let $X$ be a Fr\'{e}chet space  and and assume that $\mathcal{V} = \{ V_n\}$ is a nested sequence of linear subspaces of $X$ satisfying $\overline{V_n} \subseteq V_{n+1},$ where the closure is taken with respect to $ \| \cdot \|_F .$ Let  $d_{n, \mathcal{V}}$ be  defined as above and $\{e_n\}$ be a decreasing sequence  of positive numbers satisfying 
$$
\sum_{j=n}^{\infty} 2^{j-n} ( \delta_j+e_j) < \min \{d_{n,\mathcal{V}}, e_{n-1}\}
$$ 
with a fixed sequence of positive numbers $\delta_j$. Then, there exists $x\in X$ such that $$\rho(x, V_n)= e_n\quad \mbox{ for any} \quad n \in \mathbb{N}.$$

\end{theorem}
\begin{remark}
 Note that the condition given in the above theorem extends Borodin's condition. In the  case  when $X$ is a Banach space, we have $d_{n, \mathcal{V} }= + \infty$, and the inequality $$ \sum_{j=n}^{\infty} 2^{j-n} e_j < \mbox{min}\{d_{n, \mathcal{V}}, e_{n-1}\}$$
reduces to $$e_{n-1}> \sum_{j=n}^{\infty} 2^{j-n} e_j $$ compared with the condition $e_{n-1} > \displaystyle \sum_{j=n}^{\infty} e_j $ given in  Borodin's theorem.
\end{remark}

The idea of the proof of  the Theorem\ref{Equ} above lies in the following claims:
\renewcommand{\labelenumi}{\alph{enumi})}
\begin{enumerate}
\item $F_n = \{ v \in V_{n+1} : \rho(v,V_j)= e_j \hbox{ for } j=1,...,n\} \neq \emptyset$.

\item  $F_n$ consists of elements of the  form described as $w_n = \sum_{j=1}^n q_{j,n}$,
where $$ q_{j,n} = t_{j,n}v_j- z_{j,n},$$  with $ t_{j,n} \in [0,1],$ $ z_{j,n} \in Z_j$ ( where $Z_j$ a finite subset of $V_j$) for $j=1,...,n$ and $ v_j$ are given by the  equation
$$\sum_{j=n}^{\infty} 2^{j-n} (e_j  + \delta_j)= \rho(v_n, V_n)\geq \|v_n\| - \delta_n.$$
Moreover,$$\| q_{j,n}\| < \sum_{l=j}^{\infty}2^{l-j}(e_l+\delta_l).$$
\end{enumerate}

Using a diagonal argument, select a subsequence $\{n_k\}$ with $\|q_{j,n_k}-q_j\|\to 0$ and $q_j \in V_{j+1}$.  Fix $j_0\in \mathbb{N}$. Then, for $k\geq k_0$ and $n_k \geq j_0$, to construct $x \in X$ with $\rho(x, V_{j_0}) = e_{j_0}$,  set $s_k=\displaystyle\sum_{j=1}^{k} q_j $, show $\{s_k\}$ is Cauchy and converges to $x =\displaystyle\sum_{j=1}^{\infty}q_j$ and then show that  $\|w_{n_k}- x \| \to 0$.
Since $w_{n_k} \in F_{n_k}$, we have:

$$ \rho(w_{n_k} , V_{j_0}) = e_{j_0}\quad \mbox{for} \quad  k \geq k_0$$ hence
$$ \rho(x, V_{j_0} ) = \lim_{k} \rho(w_{n_k},V_{j_o})= e _{j_0}.$$
For the details of the proof  we refer the reader to \cite{Ak-Le2}.
\begin{theorem} [Aksoy, Lewicki \cite{Ak-Le2}]\label{main}
\label{main}
Let $X$ be a infinite-dimensional Fr\'echet space and let $\mathcal{V}=\{V_n\}$ be a nested sequence of subspaces of $ X$ such that $ \overline{V_n} \subseteq V_{n+1}$ for any $ n \in \mathbb{N}.$ 
Let $ e_n$ be a decreasing sequence of positive numbers tending to 0. Assume that 
\begin{equation}
 \label{nondegenerate}
d_{\mathcal{V}} = \inf\{ d_{n,\mathcal{V}}: n \in \mathbb{N} \} > 0.
\end{equation}
Then there exists $ n_o \in \mathbb{N}$ and $ x \in X$ such that for any $ n \geq n_o$
\begin{equation}
\label{no}
\frac{e_n}{3} \leq \rho (x,V_n) \leq 3 e_n.
\end{equation}
\end{theorem}

\begin{theorem} [Shapiro's Theorem for Fr$\acute{e}$chet spaces]
Let the assumptions of Theorem \ref{main} be satisfied. Then, there exists $ x \in X$ such that $\rho(x, V_n) \neq O(e_n).$ 
\end{theorem}
\label{Sha}
\begin{proof}
By Theorem \ref{main} applied to the sequence $ \{ \sqrt{e_n}\}$ there exists $ x\in X$ and $ n_o \in \mathbb{N}$ such that
$$
\frac{\sqrt{e_n}}{3} \leq \rho(x, V_n) \leq 3 \sqrt{e_n}
$$ 
for $n \geq n_o.$
Since $ e_n \rightarrow 0,$ it is obvious that $ \rho(x, V_n) \neq O(e_n).$
\end{proof}

\begin{theorem}[Tyuremskikh's Theorem for Fr$\acute{e}$chet spaces]
Let the assumptions of Theorem \ref {main} be satisfied. Then, there exists $ x \in X$ and $ n_o \in \mathbb{N}$ such that $\rho(x, V_n) \geq e_n$ for $ n\geq n_o.$ 
\end{theorem}

\begin{proof}
By Theorem \ref{main} applied to the sequence $ \{ 3\sqrt{e_n}\}$ there exists $ x\in X$ and $ n_o \in \mathbb{N}$ such that
$$
\sqrt{e_n} \leq \rho(x, V_n) \leq 9 \sqrt{e_n}
$$ 
for $ n\geq n_o.$
Since $ e_n \rightarrow 0,$ it is obvious that $ \rho(x, V_n) \geq \sqrt{e_n} \geq e_n$ for $ n\geq n_o.$ 
\end{proof}

\section { Bernstein Pairs}
Suppose $X$ and $Y$ are infinite dimensional Banach spaces. Let $ \mathcal{L}(X,Y)$ denote the normed space of all bounded linear  maps from $X$ to $Y$. In this section we investigate the existence of an operator $T\in \mathcal{L}(X,Y)$  whose sequence of approximation numbers $\{a_n(T)\}$ behaves like the prescribed sequence $\{d_n\}$ given in the Bernstein's Lethargy Theorem.
 
\begin{definition}
Let $X$ and $Y$ be Banach spaces. For every operator  $T \in  \mathcal{L}(X,Y)$ the n-th approximation number of $T$ is defined as:
$$ a_n(T) = \inf\{||T - S|| : \quad S \in  \mathcal{L}(X,Y)  \quad \mbox{rank}\, S < n\}.$$ 
\end{definition}

Clearly,
$$ a_n(T) = \rho (T, \mathcal{F}_n)$$ where $ \mathcal{F}_n$ is the space of all bounded linear maps from $X$ into $Y$ with rank at most $n$, and

$$ ||T|| =a_1(T) \geq a_2(T) \geq  \dots \geq 0.$$
The connection with Bernstein's Lethargy theorem  and the approximation numbers is based on the  important relation between analytical entities for  bounded linear maps such as eigenvalues and the geometrical quantities typified by approximation numbers. If we focus on compact liner maps, then there are known inequalities between the approximation numbers and the non-zero eigenvalues of such maps. For example, the well known inequality of Weyl \cite{We}, which states if the approximation numbers  form a sequence in $\ell^{p}$ for $0<p< \infty$, then so do the eigenvalues. Even though Weyl's original result is for Hilbert spaces, there is a remarkable theorem of K\"{o}nig \cite{Ko}, for the absolute value of the eigenvalues of a compact map in terms of approximation numbers. More precisely, Let $X$ be a complex Banach space and let $T$ be a compact operator on $X$. Assume that the nonzero eigenvalues of $T$ are ordered in non-increasing modulus and counted according to their (finite) multiplicity, and denote them by $\lambda_{n}(T)$.  K\"{o}nig proves that 
$$|\lambda_{n} (T )| = \lim_{m\to \infty} (a_n(T^{m} )^{ \frac{1}{m}} .$$  It is known that the approximation numbers are the largest s-numbers;  for the axiomatic theory of s-numbers we refer the reader to  \cite{Pi2}. 

  If $X$ and $Y$ are Hilbert spaces, then the sequence  of s-numbers $( s_n(T))$ or, in particular, the sequence of approximation numbers $(a_n(T))$ are the singular numbers of $T$.  Moreover if  $T$ is a compact operator, then $a_n(T)= \lambda _n( T^*T)^{1/2}$ where
$$ \lambda_1( T^*T)\geq\lambda_2( T^*T)\geq \dots \geq 0$$ are the eigenvalues of $ T^*T$ ordered as above.  We refer the reader to \cite{Carl,Pi} for general information about approximation and other s-numbers. We say  $T$ is \textit{approximable} if $\displaystyle \lim_{n \to \infty} a_n(T)=0$. Any approximable operator is compact, but the converse is not true due to the existence of Banach spaces without an approximation property \cite{Enflo}.

\begin{definition}
Two Banach spaces X and Y are said to form a Bernstein pair (BP) if for any positive monotonic null sequence $(d_n)$ there is an operator $T\in \mathcal{L}(X,Y)$
and a constant $M$ depending only on $T$ and $(d_n)$ such that 

$$d_n \leq a_n(T) \leq M d_n  \quad \quad \mbox{ for all}\quad n.$$  In this case we say that $(a_n(T))$ is equivalent to $(d_n)$ and  we write $(X, Y)$  to denote Bernstein pair and in the case $M=1$, then $(X,Y)$ is called an exact Bernstein pair.
\end{definition}
Note that by $a_n(T)$ we mean the n-th approximation number of $T$ defined above.
Note that if  $X=Y=H$ where $H$ is a  Hilbert space and if  $\{d_n\}$ is a sequence decreasing to $0$, then there exists \textit{a compact operator} $T$ on $H$ such that $$a_n(T)= d_n$$ (see \cite{Pi2}). Therefore $(H, H)$ form an exact Bernstein pair.

 All of the above mentioned desirable properties of $a_n(T)$ for operators between Banach spaces prompts the following questions:\\[.05in]
\noindent\textbf{Questions:}
Suppose  $\{d_n \}_{n\ge1}$ is a non-increasing null sequence (i.e., $d_n \searrow 0 ^+$) of positive numbers. Does there exist $T \in  \mathcal{L}(X , Y )$ such that the sequence $(a_n(T))$ behaves like the sequence $\{d_n\}$?\\[.05in]
In other words, suppose $d_n \searrow 0 ^+$, 
\begin{enumerate}
\item Does there exist $T \in  \mathcal{L} (X , Y )$ such that   $a_n(T)  \geq d_n$ for any $n$?
\item Does there exist $T \in   \mathcal{L}(X , Y )$ such that  $a_n(T)  = d_n$  for any $n$?
\item Does there exist $T \in   \mathcal{L}(X , Y )$  and a constant $M$ such that $$\displaystyle \frac{d_n}{M} \leq a_n(T) \leq M d_n$$ for any $n$?
\end{enumerate}

First results of this kind appeared  in \cite{HMR} in the context of Banach spaces with Schauder basis. 

 An operator acting on a Banach space $X$ is called an $H$-operator (see \cite{Mar}) if its spectrum is real and its resolvent satisfies
$$ ||(T-\lambda I)^{-1}|| \leq C |\mbox{Im}\lambda |^{-1}\quad \mbox{where}\quad  \mbox{Im} \lambda \neq 0 \quad$$ Here $C$ is independent of the points of the resolvent. An operator on a Hilbert space is an $H$-operator with constant $C=1$ if and only if it is a self-adjoint operator. Therefore the concept of an $H$-operator is the generalization in a Banach space  the concept of a self-adjoint operator.
In \cite{Mar} some examples and a number of properties of $H$-operators are given. 
\begin{definition}
 We define $n$-th Kolmogoroff diameter of T, $d_n(T) $, by $d_n(T)= d_n ( T(U_{X})) $ where for a bounded subset $A$ of $X$,  the $n$-th Kolmogoroff diameter of $A$
 $$ d_n(A) =\inf_{L}[ \inf \{ \epsilon > 0: \,\, A \subset \epsilon U_{X} +L \}],$$ where the infimum is taken over all at most $n$-dimensional subspaces $L$ of $Y$. 
\end{definition}
 It is clear that $d_n(T) $ and $a_n(T)$ are monotone decreasing sequences and that  
 
 $$ \lim_{n} a_n(T) =0 \quad \mbox{if and only if} \quad  T\in \mathcal{F} (X, Y) $$ 
 and
    $$ \lim_{n} d_n(T) =0 \quad \mbox{if and only if} \quad T\in \mathcal{K} (X, Y) $$ 
    where $\mathcal{K} (X, Y)$  and $ \mathcal{F} (X, Y) $is the collection of all  compact operators and finite rank maps from $X$ to $Y$.  For more on Kolmogorov diameters  and their relations to approximation numbers consult \cite{Pi2}.
\begin{theorem}[Marcus \cite{Mar}]  \label{Marcus}
If $T$ is a compact $H$-operator on a Banach space $X$, then $$ d_n(T) \leq a_n(t) \leq 2 \sqrt{2}\,\,C |\lambda_n(T) | \leq 8C(C+1) d_n(T),$$ where $C$ is the constant given in the definition of an $H$-operator and $d_n(T) $ is the $n$-th Kolmogoroff diameter of T.
\end{theorem}

Recall that  a biorthogonal system $(x_n, f_n )$  in $X$ is sequence  $\{x_n\} \subset X$ and   $\{f_n\} \subset X^*$ such that $f_i(x_j) = \delta_{ij}$.  A Schauder basis $X$ is a biorthogonal system $(x_n,f_n)$ such that for each $x\in X$, $x=\displaystyle \sum_{i=1}^{\infty} f_i(x) x_i$. 

Using  the above estimate of Marcus, Hutton, Morell and Retherford  proved the following :  
\begin{theorem}[ Hutton, Morrell and Retherford  \cite{HMR}] 
Suppose $\{d_n \}_{n\ge1}$ is a non-increasing null sequence of positive numbers and $X$ is a Banach space with Schauder basis. Then there exists $T\in B(X)$ such that $$d_m \leq a_m(T) \leq K  . d_m$$ where the constant $K$depends on $X$.
\end{theorem}
\textbf{Proof}
Suppose $(x_n)$ is a Schauder basis for $X$.
If $\{\lambda_n \}\searrow  0$, then the operator  $$ T =\displaystyle\sum_{n=1}^{\infty}  \lambda_n f_n \otimes x_n$$ is a compact $H$-operator with the sequence $\{\lambda_n\}$ as  eigenvalues. Here  $(x_n, f_n)$ is a biorthogonal system in $X$, then the result follows from Theorem \ref {Marcus}. For further details consult \cite{HMR}.\\

These results are sharpened in \cite{Ak-Le} as follows:

\begin{theorem}[Aksoy,  Lewicki  \cite{Ak-Le}] 
Suppose  $\{d_n \}_{n\ge1}$ is a non-increasing null sequence of positive numbers. In each of the following cases there exists $T\in B(X,Y)$ such that 
$$a_m(T) = d_m \quad \mbox{for each} \quad m$$
\renewcommand{\labelenumi}{\alph{enumi})}
\begin{enumerate}
\item $X$ and $Y$ has 1-unconditional basis.\\ (For example $X=Y=\ell_p$ for $1\leq p < \infty$ and $X=Y=c_0$).
\item $X\in \{c_0, \ell_{\infty}\}$ and $Y$ has 1-symmetric basis.
\item $Y=\ell_1$, $X$ has 1-symmetric basis.
\end{enumerate}
\end{theorem}
In \cite{Ak-Le}, among other things, it was proved that if we assume $(X,Y)$ is a Bernstein pair with respect to $\{a_n\}$ and suppose that a Banach space $W$ contains an isomorphic and  complementary copy of $X$, and a Banach space $V$ contains an isomorphic copy of $Y$, then $(W, V)$ is a Bernstein Pair with respect to $\{a_n\}$. This implies that  there are some natural pairs of Banach spaces that form a Bernstein pair as shown in the following examples.

\begin{example}
For $1< p < \infty$ and $1 \leq q < \infty$, the couple $(L_p[0,1], L_q[0,1])$ form a Bernstein pair.
\end{example}
The above example follows from the fact that $(\ell_2, \ell_2)$ is a Bernstein pair with respect to $\{a_n\}, $ \cite{Pi} ,and the fact that for every $p$, $1\leq p < \infty$, $L_p[0,1]$ contains a subspace isomorphic to $\ell_2$ and complemented in $L_p[0,1]$ for $p> 1$ (see \cite{Wo}, page 85).  If one assumes that the Banach spaces  $X$ and $Y$ have certain properties then one can generate Bernstein pairs as illustrated in the following  ``conditional" example.

\begin{example}
\renewcommand{\labelenumi}{\alph{enumi})}
\begin{enumerate}
\item If $(\ell_{\infty},\ell_{\infty}) $ is BP, then $(X,Y)$ is BP provided that both $X$ and $Y$ contains an isomorphic copy of $\ell_{\infty}$ (Phillips Theorem).
\item If $(c_0,c_0)$  is a BP, then $(X,Y)$ is BP provided that both $X$ and $Y$ each contain an isomorphic copy of $c_0$ and $X$ is separable  (Sobczyk's Theorem).
\item if $(\ell_1, \ell_1)$ is a BP, then $(X,Y)$ is BP provided that both $X$ and $Y$ each contain an isomorphic copy of $\ell_1$ and $X$ is a non reflexive subspace of $L_1[0,1]$ (Pelczynski-Kade\v c Theorem). 
\end{enumerate}
\end{example}
 For the statements of Phillips, Sobczyk and Pelczynski-Kade\v c theorems we refer the reader to \cite {Di}. The most recent result on the rate of decay of approximation numbers can be found in \cite{TO}. It is shown there that  for  $X$ and $Y$, infinite dimensional  Banach spaces, and $\{d_n \}_{n\ge1}$  a non-increasing null sequence, there exists  an approximable $T: X \to Y$ such that 
$$ ||T|| \leq 2d_1 \quad \mbox{and}\quad d_m/9 \leq a_m(T) \leq 3d_{[m/4]} \quad \mbox{for any } m,$$
where $[m/4]$ denotes the integer part of $m/4$.
It is also worth mentioning that in \cite{KR} operators with prescribed eigenvalue sequences are constructed.


 \noindent

\mbox{Asuman G\"{u}ven AKSOY}\\
{~~~~~~~}\mbox{Claremont McKenna College}\\
{~~~~~~~}\mbox{Department of Mathematics}\\
{~~~~~~~}\mbox{Claremont, CA  91711, USA} \\
{~~~~~~~}\mbox{E-mail: aaksoy@cmc.edu} 

\end{document}